\documentclass[12pt]{amsart}

\usepackage{amsmath, amssymb, amscd, paralist,tabularx,supertabular,amsmath, verbatim, amsthm, amssymb, mathrsfs, manfnt,times,latexsym,amscd,graphicx, color, hyperref}
\usepackage{amsmath, amsthm, amssymb}
\usepackage{paralist}  
\usepackage{manfnt}    
\usepackage{url}
\usepackage[all]{xy}

\usepackage{fullpage}

\DeclareMathOperator{\coarea}{coarea}

\DeclareMathOperator{\Gal}{Gal}

\DeclareMathOperator{\M}{M}

\DeclareMathOperator{\PGL}{PGL}

\DeclareMathOperator{\PSL}{PSL}
\DeclareMathOperator{\Ram}{Ram}

\DeclareMathOperator{\SL}{SL}

\DeclareMathOperator{\tr}{tr}

\newcommand{\Q}{\mathbb Q}
\newcommand{\R}{\mathbb R}
\newcommand{\Z}{\mathbb Z}

\newcommand{\bfH}{\mathbf H}

\newcommand{\frakp}{\mathfrak{p}}
\newcommand{\frakq}{\mathfrak{q}}

\newcommand{\calG}{\mathcal{G}}

\newcommand{\calK}{\mathcal{K}}

\newcommand{\calO}{\mathcal{O}}

\numberwithin{equation}{section}

\theoremstyle{remark}

\newtheorem{example}[equation]{Example}

\newtheorem{rmk}[equation]{Remark}

\theoremstyle{plain}
\newtheorem{thm}[equation]{Theorem}
\newtheorem{theorem}[equation]{Theorem}

\newtheorem{lemma}[equation]{Lemma}
\newtheorem{cor}[equation]{Corollary}

\title{Counting problems for geodesics on arithmetic hyperbolic surfaces}

\author{Benjamin Linowitz}
\address{Department of Mathematics\\Oberlin College\\Oberlin, OH 44074}
\email{benjamin.linowitz@oberlin.edu}

\begin{document}

\begin{abstract} 
It is a longstanding problem to determine the precise relationship between the geodesic length spectrum of a hyperbolic manifold and its commensurability class. A well known result of Reid, for instance, shows that the geodesic length spectrum of an arithmetic hyperbolic surface determines the surface's commensurability class. It is known, however, that non-commensurable arithmetic hyperbolic surfaces may share arbitrarily large portions of their length spectra. In this paper we investigate this phenomenon and prove a number of quantitative results about the maximum cardinality of a family of pairwise non-commensurable arithmetic hyperbolic surfaces whose length spectra all contain a fixed (finite) set of nonnegative real numbers.
\end{abstract}

\maketitle

\section{Introduction}

Let $M$ be an orientable hyperbolic manifold (or orbifold) with finite volume. The {\it length spectrum} of $M$ is defined to the set of all lengths of closed geodesics in $M$. Further, two manifolds are said to be {\it commensurable} if they share an isometric finite-sheeted covering. Commensurability is an equivalence relation, and the {\it commensurability class} of $M$ is the equivalence class containing $M$.

One of the earliest results concerning the relationship between the length spectrum of a hyperbolic manifold and its commensurability class is due to Reid \cite{R} and shows that if two arithmetic hyperbolic $2$-manifolds have the same length spectra then they are necessarily commensurable. This was later extended to arithmetic hyperbolic $3$-manifolds by Chinburg-Hamilton-Long-Reid \cite{CHLR}. It turns out that one does not need the entire length spectrum in order to force commensurability in these cases. In \cite{LMPT}, Linowitz, McReynolds, Pollack and Thompson showed that two arithmetic hyperbolic $3$-manifolds of volume at most $V$ whose length spectra coincide for all geodesic lengths less than $c\cdot \left(\exp(\log V^{\log V})\right)$ are commensurable, where $c>0$ is an absolute constant. A similar result was proven for arithmetic hyperbolic surfaces.

Although a number of authors have addressed the relationship between the length spectrum of a hyperbolic manifold and its commensurability class in the arithmetic setting, to our knowledge the only papers that consider the non-arithmetic setting are those of Millichap \cite{Mi} and Futer and Millichap \cite{FM}, where families of non-commensurable $3$-manifolds having the same volume and the same $n$ shortest geodesic lengths were constructed.
 
The past ten years have seen a number of papers considering this problem for more general locally symmetric spaces. Lubotzky, Samuels and Vishne \cite{LSV}, for instance, have constructed non-commensurable arithmetic manifolds with universal cover the symmetric space associated to $\PGL_n(\mathbb R)$ (for $n\geq 3$) having the same length spectra. More generally, Prasad and Rapinchuk \cite{PR} have considered locally symmetric spaces $\mathfrak X_\Gamma=\calK\backslash \calG / \Gamma$ where $\calG=G(\mathbb R)$ is the Lie group associated to a connected semi-simple real algebraic subgroup $G$ of $\SL_n$, $\calK$ is a maximal compact subgroup of $\calG$ and $\Gamma$ is a discrete torsion-free subgroup of $\calG$. In particular they showed that there exist pairs of non-commensurable locally symmetric spaces $\mathfrak X_{\Gamma_1}$ and $\mathfrak X_{\Gamma_2}$ with the same length spectra only if $G$ is of type $A_n (n>1)$, $D_{2n+1} (n\geq 1)$, $D_4$ or $E_6$.

In this paper we focus on hyperbolic surfaces and prove a variety of results which quantify the extent to which two non-commensurable hyperbolic surfaces may contain many geodesic lengths in common. Because we will be considering arithmetic hyperbolic surfaces, we briefly recall what it means for a hyperbolic surface to be arithmetic. Given a discrete subgroup $\Gamma$ of $\PSL_2(\mathbb R)$, the {\it commensurator} of $\Gamma$ is the set \[\mathrm{Comm}(\Gamma)=\{g\in \PSL_2(\mathbb R) : \Gamma\text { and }g\Gamma g^{-1}\text{ are commensurable }\}.\] The celebrated Margulis dichotomy \cite{Ma} states that $\Gamma$ is arithmetic if and only if $\Gamma$ has infinite index in $\mathrm{Comm}(\Gamma)$. An alternative characterization of arithmeticity defines a hyperbolic surface to be arithmetic if and only if it is commensurable with a hyperbolic surface of the form ${\bf H}^2/\Gamma_\mathcal O$. Here ${\bf H}^2$ denotes the hyperbolic plane and $\Gamma_\mathcal O$ is a group constructed from a maximal order in a quaternion algebra defined over a totally real field (we will review the construction of $\Gamma_\mathcal O$ in Section \ref{section:arithmetic}). We note that an arithmetic hyperbolic surface is called {\it derived from a quaternion algebra} if its fundamental group is contained in a group of the form $\Gamma_\mathcal O$.

We now define a counting function whose behavior will be studied throughout this paper. Given a set $S=\{\ell_1,\dots,\ell_r\}$ of nonnegative real numbers we define $\pi(V, S)$ to be the maximum cardinality of a collection of pairwise non-commensurable arithmetic hyperbolic surfaces derived from quaternion algebras, each of which has volume less than $V$ and length spectrum containing $S$.

The function $\pi(V,S)$ was previously studied in \cite[Theorem 4.10]{LMPT}, where it was shown that if $\pi(V,S)\to\infty$ as $V\to\infty$ then there exist integers $1\leq a,b\leq |S|$ and constants $c_1,c_2>0$ such that \[c_1\frac{V}{\log V^{1-\frac{1}{2^{a}}}} \leq \pi(V,S) \leq c_2\frac{V}{\log V^{1-\frac{1}{2^{b}}}}\] for all sufficiently large $V$.

The first result of this paper considers the asymptotic behavior of $\pi(V,S)$ in short intervals and provides a lower bound on the number of arithmetic hyperbolic surfaces which are pairwise non-commensurable, have length spectra containing $S$ and volume contained in an interval of the form $[V,V+W]$.

\begin{theorem}\label{theorem:shortintervals}
Fix a finite set $S$ of nonnegative real numbers for which $\pi(V,S)\to\infty$ as $V\to\infty$. Let $r$ be the cardinality of $S$ and define $\theta=\frac{8}{3}$ if $r=1$ and $\theta=\frac{1}{2^r}$ otherwise. If $\epsilon>0$ and $V^{1-\theta+\epsilon} < W < V$ then as $V\to\infty$ we have \[\pi(V+W,S)-\pi(V,S) \geq \frac{1}{2^r}\cdot \frac{W}{\log V}.\]
\end{theorem}

The assumption that $\pi(V,S)\to\infty$ as $V\to\infty$ is necessary in Theorem \ref{theorem:shortintervals} because of the existence of sets $S$ for which the function $\pi(V,S)$ is non-zero yet constant for all sufficiently large $V$. The remainder of this paper is devoted to a careful analysis of the situation in which $\pi(V,S)$ is eventually constant.

In Lemma \ref{lemma:sameinvariants} we will show that if $S$ is such that $\pi(V,S)>0$ for sufficiently large $V$ then every arithmetic hyperbolic surface with length spectrum containing $S$ must have the same invariant trace field (see Section \ref{section:arithmetic} for a definition). In the following theorem we will denote this common invariant trace field by $k$.

\begin{theorem}\label{theorem:finiteness}
Suppose that for some fixed (finite) set $S$ of nonnegative real numbers the function $\pi(V,S)$ is eventually constant and greater than zero. There exist integers $\ell,m,n$ with $\ell\in \{0,1\}$, $m\in\{1,[k:\mathbb Q]\}$ and $n\geq0$ such that \[\lim_{V\to\infty} \pi(V,S)=m 2^{n}-\ell.\] Furthermore, $\ell=0$ whenever $k$ has narrow class number one.
\end{theorem}

The case in which $k=\mathbb Q$ is especially nice, as this field has narrow class number one and of course satisfies $[k:\mathbb Q]=1$. Theorem \ref{theorem:finiteness} therefore immediately implies:

\begin{cor}\label{cor:2n}
Suppose that for some fixed (finite) set $S$ of nonnegative real numbers the function $\pi(V,S)$ is eventually constant. If $\mathbb Q$ is the invariant trace field associated to $S$ then there is an integer $n\geq 0$ such that \[\lim_{V\to\infty} \pi(V,S)=2^{n}.\]
\end{cor}

As a complement to Corollary \ref{cor:2n} we prove the following theorem which shows that for every integer $n\geq 0$ one can find a set $S$ such that $\lim_{n\to\infty} \pi(V,S)=2^n$.

\begin{theorem}\label{theorem:existence}
For every integer $n\geq 0$ there exists a set $S$ of nonnegative real numbers such that \[\lim_{V\to\infty} \pi(V,S)=2^{n}.\]
\end{theorem}

Our proofs are for the most part number theoretic and make extensive use of the correspondence between lengths of closed geodesics on arithmetic hyperbolic surfaces and algebraic integers in quadratic subfields of certain quaternion algebras. Of particular importance are Borel's formula for the area of an arithmetic hyperbolic surface \cite{Borel}, a ``selectivity'' theorem for embeddings of commutative orders into quaternion orders due to Chinburg and Friedman \cite{CF-S}, as well as a version of the Chebotarev density theorem in short intervals due to Balog and Ono \cite{BO}. 


\section{Quaternion algebras and quaternion orders}

Let $k$ be a number field with ring of integers $\mathcal O_k$. A quaternion algebra over $k$ is a central simple $k$-algebra of dimension $4$. Equivalently, a quaternion algebra over $k$ is a $4$-dimensional $k$-vector space with basis $\{1,i,j,ij\}$ such that $i^2,j^2\in k^*$, $ij=-ji$ and such that every element of $k$ commutes with $i$ and $j$. Suppose that $B$ is a quaternion algebra over $k$. Given a prime $\frakp$ of $k$, we define the completion $B_\frakp$ of $B$ at $\frakp$ as $B_\frakp=B\otimes_k k_\frakp$. The classification of quaternion algebras over local fields shows that if $B_\frakp$ is not a division algebra then $B_\frakp\cong \M_2(k_\frakp)$. If $B_\frakp$ is a division algebra we say that $\frakp$ {\it ramifies} in $B$. Otherwise we say that $\frakp$ is {\it unramified} or {\it split} in $B$. The set of primes of $k$ (finite or infinite) which ramify in $B$ is denoted $\Ram(B)$. We denote by $\Ram_f(B)$ (respectively $\Ram_\infty(B)$) the set of finite (respectively infinite) primes of $k$ which ramify in $B$. The set $\Ram(B)$ is known to be finite of even cardinality. Conversely, given any finite set $T$ of primes of $k$ (which are either finite or else real) having even cardinality there exists a unique (up to isomorphism) quaternion algebra $B$ over $k$ for which $\Ram(B)=T$. Note that $B$ is a division algebra if and only if $\Ram(B)\neq\emptyset$.

Given a quaternion algebra $B$ over a number field $k$, we define a {\it quaternion order} to be a subring of $B$ which is also finitely generated as an $\mathcal O_k$-module and contains a basis for $B$ over $k$. A quaternion order is called a maximal order if it is not properly contained in any other quaternion order.


\section{Arithmetic hyperbolic surfaces and their closed geodesics}\label{section:arithmetic}

Let $k$ be a totally real field of degree $n_k$ with absolute value of discriminant $d_k$ and Dedekind zeta function $\zeta_k(s)$. Let $B$ be a quaternion algebra over $k$ which is unramified at a unique real place $\nu$ of $k$. This gives us an identification $B\otimes_k k_\nu\cong \M_2(\mathbb R)$. Let $\mathcal O$ be a maximal order of $B$ and $\mathcal O^1$ be the multiplicative subgroup of $\mathcal O^*$ consisting of those elements with reduced norm one. We denote by $\Gamma_\mathcal O$ the image of $\mathcal O^1$ in $\PSL_2(\mathbb R)$. It was shown by Borel \cite{Borel} that $\Gamma_\mathcal O$ is a discrete subgroup of $\PSL_2(\mathbb R)$ whose coarea is given by the formula:

\begin{equation}\label{equation:volumeformula}
\coarea(\Gamma_{\calO})=\frac{8\pi d_k^{\frac{3}{2}}\zeta_k(2)}{(4\pi^2)^{n_k}}\prod_{\frakp\in\Ram_f(B)}\left(N(\frakp)-1\right).
\end{equation}

We define an {\it arithmetic Fuchsian group} to be a discrete subgroup of $\PSL_2(\mathbb R)$ which is commensurable with a group of the form $\Gamma_\mathcal O$. An arithmetic Fuchsian group is {\it derived from a quaternion algebra} if it is contained in a group of the form $\Gamma_\mathcal O$. Although not every arithmetic Fuchsian group $\Gamma$ is derived from a quaternion algebra, it is known that the subgroup $\Gamma^2$ of $\Gamma$ generated by squares of elements of $\Gamma$ is always derived from a quaternion algebra \cite[Chapter 8]{MR}. An {\it arithmetic hyperbolic surface} is a hyperbolic surface of the form ${\bf H}^2/\Gamma$ where $\Gamma$ is an arithmetic Fuchsian group. We will say that an arithmetic hyperbolic surface is derived from a quaternion algebra if its fundamental group $\Gamma$ is derived from a quaternion algebra.

Suppose that $\Gamma$ is an arithmetic Fuchsian group. The {\it trace field} of $\Gamma$ is the field $\mathbb Q(\tr\gamma : \gamma\in \Gamma)$. It follows from the Mostow Rigidity Theorem that this trace field is a number field. Although it turns out that the trace field of an arithmetic Fuchsian group is not an invariant of the commensurability class, it can be shown that the {\it invariant trace field} $\mathbb Q(\tr\gamma^2 : \gamma\in \Gamma)$ is a commensurability class invariant. We will denote the invariant trace field of $\Gamma$ by $k\Gamma$. 

We will now define a quaternion algebra over $k\Gamma$. Let \[B\Gamma = \left\{\sum b_i\gamma_i : b_i\in k\Gamma, \gamma_i\in \Gamma\right\}\] where only finitely many of the $b_i$ are non-zero. We may define multiplication in $B\Gamma$ in the obvious manner: $(b_1\gamma_1)\cdot (b_2\gamma_2)=(b_1b_2)(\gamma_1\gamma_2)$. The algebra $B\Gamma$ is a quaternion algebra over $k\Gamma$ which we call the {\it invariant quaternion algebra} of $\Gamma$.

Suppose that $\Gamma_1, \Gamma_2$ are arithmetic Fuchsian groups. It was shown by Maclachlan and Reid \cite[Chapter 8.4]{MR} that the surfaces ${\bf H}^2/\Gamma_1$ and ${\bf H}^2/\Gamma_2$ are commensurable in the wide sense if and only if $k\Gamma_1\cong k\Gamma_2$ and $B\Gamma_1\cong B\Gamma_2$.

Let $\Gamma$ be an arithmetic Fuchsian group and $\gamma\in\Gamma$ be a hyperbolic element. Let $\lambda=\lambda_\gamma$ be an eigenvalue of a preimage of $\gamma$ in $\SL_2(\R)$ for which $|\lambda|>1$. Then $\lambda$ is well-defined up to multiplication by $\pm 1$. The axis of $\gamma$ in $\bfH^2$ projects to a closed geodesic on $\bfH^2/\Gamma$ of length $\ell=\ell(\gamma)$ where $\cosh(\ell/2)=\pm\tr(\gamma)/2$.


\section{Quaternion algebras with specified maximal subfields}

In this section we prove a variety of results concerning quaternion algebras admitting embeddings of a fixed set of quadratic fields. These results will play an important role in the proofs of this paper's main theorems.

\begin{example}
Consider the three real quadratic fields $\Q(\sqrt{3}),\Q(\sqrt{17}), \Q(\sqrt{51})$. The only primes that do not split in any of these fields are $3$ and $17$. It follows that if $B$ is a quaternion division algebra over $\Q$ which admits embeddings of $\Q(\sqrt{3}),\Q(\sqrt{17})$ and $\Q(\sqrt{51})$ then $B$ is the unique quaternion division algebra over $\Q$ with $\Ram(B)=\{3,17\}$. The quaternion algebra $\M_2(\mathbb Q)$ also admits embeddings of these quadratic fields, hence there are (up to isomorphism) two quaternion algebras over $\mathbb Q$ which admit embeddings of $\Q(\sqrt{3}),\Q(\sqrt{17})$ and $\Q(\sqrt{51})$.
\end{example}

\begin{theorem}\label{theorem:csatheorem1}
If $L_1,\dots, L_r$ is a collection of quadratic extensions of a number field $k$ with the property that only finitely many quaternion algebras over $k$ admit embeddings of the $L_i$ then the number of isomorphism classes of quaternion algebras over $k$ which admit embeddings of all of the $L_i$ is $2^n$ for some $n\geq 0$.
\end{theorem}
\begin{proof}
Let $k$ be a number field and $L_1,\dots, L_r$ be a collection of quadratic extensions of $k$ such that there are only finitely many isomorphism classes of quaternion algebras over $k$ admitting embeddings of all of the $L_i$. We claim that all but finitely many primes of $k$ split in at least one of the $L_i$. Indeed, suppose to the contrary that $\mathfrak p_1, \mathfrak p_2,\dots$ are distinct primes of $k$ which do not split in any of the $L_i$. Because a quaternion algebra over $k$ admits an embedding of a quadratic extension $L/k$ if and only if no prime of $k$ which ramifies in the algebra splits in $L/k$, it follows that the (mutually non-isomorphic) quaternion algebras \[\{ B_i : \Ram(B_i)=\{\mathfrak p_i, \mathfrak p_{i+1}\}\}\]  each admit embeddings of all of the $L_i$, giving us a contradiction which proves our claim.

We have shown that all but finitely many primes (finite of infinite) of $k$ split in at least one of the $L_i$. Let $S=\{\mathfrak p_1,\dots, \mathfrak p_m\}$ be the primes of $k$ not splitting in any of the $L_i$. On the one hand there are precisely $2^{m-1}$ subsets of $S$ with an even number of elements, each of which corresponds to a unique quaternion algebra (the algebra which is ramified precisely at the primes in this subset). Of these algebras, $2^{m-1}-1$ are division algebras; the remaining algebra is $\M_2(k)$ and corresponds to the empty subset of $S$. On the other hand, if $B$ is a quaternion algebra over $k$ which admits embeddings of $L_1,\dots, L_r$ then the only primes which may ramify in $B$ are those lying in $S$. It follows that $\Ram(B)\subseteq S$. Because the set $\Ram(B)$ is non-empty and determines the isomorphism class of $B$, the theorem follows.
\end{proof}

The following corollary to Theorem \ref{theorem:csatheorem1} considers a similar counting problem, though with the caveat that the quaternion algebras being considered are required to have a prescribed archimedean ramification behavior which will be necessary in our geometric applications.

\begin{cor}\label{cor:csacor}Let $k$ be a number field of signature $(r_1,r_2)$ with $r_1>0$ and $L_1,\dots, L_r$ be a collection of quadratic extensions of $k$ such that only finitely many quaternion algebras over $k$ admit embeddings of the $L_i$. There is a nonnegative integer $n$ such that the number of quaternion algebras over $k$ which admit embeddings of all of the $L_i$ and are unramified at a unique real place of $k$, if nonzero, is equal to $m2^n$ for some integer $m\in\{1,r_1\}$.
\end{cor}
\begin{proof} 
We may assume that there exists at least one quaternion algebra $B$ over $k$ which admits embeddings of all of the $L_i$ and is split at a unique real place of $k$, as otherwise the total number of algebras we are counting is $0$. Suppose that the unique real place of $k$ at which $B$ is split is $\nu$. If $\omega\neq \nu$ is a real place of $k$ then $\omega$ ramifies in $B$, hence $\omega$ does not split in any of the extensions $L_i/k$ (since no place of $k$ which ramifies in a quaternion algebra over $k$ may split in a quadratic extension of $k$ which embeds into the quaternion algebra). We now have two cases to consider.

The first case is that $\nu$ does not split in any of the extensions $L_i/k$. In this case no real place of $k$ splits in any of the extensions $L_i/k$. Fix a real place $\nu'$ of $k$. We will count the number of quaternion algebras over $k$ which admit embeddings of all of the $L_i$ and are split at $\nu'$ and no other real places of $k$. The proof of Theorem \ref{theorem:csatheorem1} shows that all but finitely many primes (finite or infinite) of $k$ split in at least one of $L_1,\dots, L_r$. Let $S=\{\frakp_1,\dots,\frakp_m\}$ be the set of all primes of $k$ which do not split in any of these extensions. Note that we have already shown that in the case we are considering $S$ contains all real places of $k$. A quaternion algebra $B$ over $k$ is ramified at all real places of $k$ not equal to $\nu'$, split at $\nu'$ and admits embeddings of $L_1,\dots, L_r$ if and only if \[\Ram(B)= \{\omega : \omega \text{ is a real place of $k$ not equal to $\nu'$}\} \bigcup S'\] for some subset $S'$ of $S$ containing only finite primes and whose cardinality ensures that $\Ram(B)$ has an even number of elements. The number of such subsets is $2^n$ for some integer $n\geq 0$, hence there are a total of $r_1 2^n$ quaternion algebras over $k$ which are split at a unique real place of $k$ and which admit embeddings of all of the $L_i$ (since there are $r_1$ choices for $\nu'$).

Now consider the case in which $\nu$ splits in one of the extensions $L_i/k$. In this case a quaternion algebra over $k$ admits embeddings of all of the $L_i$ only if $\nu$ does not ramify in the quaternion algebra. Because we are counting quaternion algebras which are ramified at all but one real place of $k$, it must be the case that all of the quaternion algebras we are counting are split at $\nu$ and at no other real places of $k$. That there is a nonnegative integer $n$ such that there are $2^n$ quaternion algebras which are split at $\nu$ and no other real place of $k$ and which admit embeddings of all of the $L_i$ now follows from the argument that was used in the previous case.
\end{proof}

\begin{theorem}\label{theorem:csatheorem2}
Let $n\in\Z$ with $n\geq 0$. For every number field $k$ there exist quadratic extensions $L_1,\dots, L_r$ of $k$ such that there are precisely $2^n-1$ isomorphism classes of quaternion division algebras over $k$ which admit embeddings of all of the $L_i$.
\end{theorem}
\begin{proof}
We begin by considering the case in which $k=\Q$. Let $p_1$ be a prime satisfying $p_1\equiv 1 \pmod{8}$ and define $L_1=\Q(\sqrt{p_1})$. Let $p_2,\dots, p_{m}$ be primes which satisfy $p_i\equiv 1\pmod{8}$ and which are all inert in $L_1/\Q$. Define $L_2=\Q(\sqrt{p_1p_2\cdots p_{m}})$ and $L_3=\Q(\sqrt{p_2\cdots p_{m}})$.

Let $d_1,d_2, d_3$ denote the discriminants of $L_1,L_2,L_3$. A prime $p$ splits in the extension $L_i/\Q$ if and only if the Kronecker symbol $\left(\frac{d_i}{p}\right)=1$, is inert in the extension if and only if $\left(\frac{d_i}{p}\right)=-1$ and ramifies if and only if $\left(\frac{d_i}{p}\right)=0$. Moreover, as $\left(\frac{ab}{p}\right)=\left(\frac{a}{p}\right)\left(\frac{b}{p}\right)$ for all positive integers $a,b$, we have the identity \[ \left(\frac{d_1}{p}\right)\left(\frac{d_2}{p}\right)\left(\frac{d_3}{p}\right)=\left(\frac{d_1d_2d_3}{p}\right)=\left(\frac{\left(p_1p_2\cdots p_m\right)^2}{p}\right)=1.\] This shows that every prime $p$ not lying in the set $\{p_1,\dots, p_m\}$ must split in one of the extensions $L_i/\Q$. While a prime $p_i$ with $i>1$ is inert in $L_1/\Q$ and ramifies in $L_2/\Q$ and $L_3/\Q$, quadratic reciprocity implies that the prime $p_1$ will split in $L_3/\Q$ if and only if $m$ is odd. 

Let $L_4$ be a real quadratic field in which the prime $p_1$ splits and in which $p_2,\dots, p_m$ are all inert. It now follows from the previous paragraph that that every prime not in $\{p_2,\dots, p_m\}$ splits in at least one of the quadratic fields $\{L_1,\dots, L_4\}$. If $B$ is a quaternion division algebra over $\Q$ into which $L_1,\dots, L_4$ all embed then the set $\Ram(B)$ of primes at which $B$ is ramified is a nonempty set of even cardinality which satisfies $\Ram(B)\subseteq \{p_2,\dots, p_m\}$. Conversely, every nonempty subset of $\{p_2,\dots,p_m\}$ with even cardinality defines a unique quaternion division algebra over $\Q$ into which the quadratic fields $\{L_1,\dots, L_4\}$ all embed. As there are precisely $2^{m-2}-1$ such subsets, setting $m=n+2$ proves the theorem in the case that $k=\Q$.

We now consider the general case in which $k$ is an arbitrary number field. Let $L_1,\dots, L_4$ be quadratic fields as above, though with the additional restrictions that $L_i\cap k=\Q$ for $i=1,\dots 4$ and that all of the primes in the set $\{p_2,\dots,p_m\}$ split completely in $k/\Q$. Let $p\not\in\{p_2,\dots,p_m\}$ be a rational prime and $\mathfrak p$ be a prime of $k$ lying above $p$. Then for $i=1,\dots,4$ the prime $\mathfrak p$ splits in the quadratic extension $kL_i/k$, where $kL_i$ is the compositum of $k$ and $L_i$. Also, if $q \in \{p_2,\dots,p_m\}$ and $\mathfrak q$ is a prime of $k$ lying above $q$ then $\mathfrak q$ is inert in $kL_i/k$ for $i=1,\dots,4$. Both of these assertions follow from standard properties of the Artin symbol \cite[Chapter X]{Lang-ANT} and the fact that $\Gal(kL_i/k)$ is isomorphic to $\Gal(L_i/\Q)$ via restriction to $L_i$. It follows that all but finitely many primes of $k$ split in at least one of the extensions $kL_1, \dots, kL_4$ and that there are at least $m-1$ primes of $k$ which do not split in any of the $kL_i$. By considering a fifth quadratic extension of $k$ in which $m-1$ of these primes are inert and the remainder of the primes split, we obtain five quadratic extensions of $k$ with the property that all but $m-1$ primes (finite or infinite) of $k$ split in at least one of these extensions. The theorem now follows, as it did in the $k=\Q$ case, from the correspondence between quaternion division algebras over $k$ admitting embeddings of these five quadratic extensions and even order subsets of these $m-1$ primes.
\end{proof}

\begin{rmk}\label{totallyreal}
Because it will be important in the proof of Theorem \ref{theorem:existence}, we remark that in the case that $k=\Q$, the quadratic fields furnished by Theorem \ref{theorem:csatheorem2} may all be assumed to be totally real. This follows immediately from the proof of Theorem \ref{theorem:csatheorem2}.
\end{rmk}

\section{Selectivity in quaternion algebras}

Let $k$ be a number field, $B$ be a quaternion algebra over $k$ which admits embeddings of the quadratic extensions $L_1,\dots, L_r$ of $k$. For each $i=1,\dots, r$, fix a quadratic $\mathcal O_k$-order $\Omega_i\subset L_i$. We would like to determine which maximal orders of $B$ contain conjugates of {\it all} of the quadratic orders $\Omega_i$. In the case that $r=1$ this problem was solved by Chinburg and Friedman \cite[Theorem 3.3]{CF-S}. Because of our interest in arithmetic hyperbolic surfaces and their invariant quaternion algebras, we are primarily interested in the case that $k$ is totally real and $B$ is unramified at a unique real place of $k$.

\begin{thm}[Chinburg and Friedman]\label{thm:CF}
Let $B$ be a quaternion algebra over a number field $k$, $\Omega\subset B$ be a commutative $\mathcal O_k$-order and assume that $B$ is unramified at some real place of $k$. Then every maximal order of $B$ contains a conjugate (by $B^*$) of $\Omega$, except when the following three conditions hold:
\begin{enumerate}[(1)]
\item $\Omega$ is an integral domain and its quotient field $L\subset B$ is a quadratic extension of $k$.
\item The extension $L/k$ and the algebra $B$ are unramified at all finite places and ramify at exactly the same (possibly empty) set of real places of $k$.
\item All prime ideals of $k$ dividing the relative discriminant ideal $d_{\Omega/\mathcal O_k}$ of $\Omega$ are split in $L/k$.
\end{enumerate}
Suppose now that (1), (2) and (3) hold. Then $B$ has an even number of conjugacy classes of maximal orders and the maximal orders containing some conjugate of $\Omega$ make up exactly half of these conjugacy classes.
\end{thm}

\begin{rmk}We note that Chinburg and Friedman actually prove a stronger result which shows exactly which conjugacy classes of maximal orders have representatives admitting embeddings of $\Omega$.\end{rmk}

\begin{thm}\label{thm:selectivity}
Let $k$ be a totally real number field and $L_1,\dots, L_r$ be quadratic extensions of $k$. For each $i=1,\dots, r$ let $\Omega_i$ be a quadratic $\mathcal O_k$-order contained in $L_i$. Suppose that there exists a quaternion algebra over $k$ which is unramified at a unique real place of $k$ and into which all of the $L_i$ embed. Then with one possible exception, every quaternion algebra over $k$ which is unramified at a unique real place of $k$ and into which all of the $L_i$ embed has the property that every maximal order of the quaternion algebra contains conjugates of all of the $\Omega_i$. Furthermore, this exceptional quaternion algebra does not exist if the narrow class number of $k$ is equal to one.
\end{thm}
\begin{proof}
Suppose first that $k$ has narrow class number one and that $B$ is a quaternion algebra over $k$ which is unramified at a unique real place of $k$ and in which all of the $L_i$ embed. It was shown in \cite[Proposition 5.4]{L-S} that if $\mathcal R$ is an order of $B$ then there is an extension $k(\mathcal R)$ of $k$ with the property that if $L_i\not\subset k(\mathcal R)$ then every order in the genus of $\mathcal R$ admits an embedding of $\Omega_i$. By the Skolem-Noether theorem, this is equivalent to the statement that $\mathcal R$ contains a conjugate of $\Omega_i$. Moreover, it was shown in \cite[Section 3]{L-S} that the conductor of $k(\mathcal R)$ is divisible only by primes which divide the level ideal of $\mathcal R$. In the case we are considering, $\mathcal R$ is a maximal order. Therefore its level ideal is trivial and the genus of $\mathcal R$ is simply the set of all maximal orders of $B$. It follows that $k(\mathcal R)$ is contained in the narrow class field of $k$. As $k$ has narrow class number one, this means that $k(\mathcal R)=k$, hence \cite[Section 3]{L-S} shows that every maximal order of $B$ contains conjugates of all of the $\Omega_i$.

We now prove the first statement of the theorem. If $k=\mathbb Q$ then $k$ has narrow class number one and we are done by the previous paragraph. We may therefore assume that $k\neq \mathbb Q$. Note that because $k$ is totally real and not equal to $\mathbb Q$, it follows that $k$ has at least two real places. By hypothesis there exists a quaternion algebra $B$ over $k$ which is unramified at a unique real place of $k$ and into which all of the $L_i$ embed. Denote by $\nu$ the real place of $k$ which is unramified in $B$. If $\omega\neq \nu$ is another real place of $k$ then $\omega$ ramifies in $B$, hence ramifies in all of the extensions $L_i/k$, as otherwise the $L_i$ would not all embed into $B$.

Let $B'$ be a quaternion algebra over $k$ which admits embeddings of all of the $L_i$ and which is unramified at a unique real place of $k$. Suppose that $B'$ and one of the extensions, say $L_i$, satisfy condition (2) in Theorem \ref{thm:CF}. We have already seen that every real place $\omega$ of $k$ not equal to $\nu$ ramifies in $L_i$. Because $B'$ and $L_i$ satisfy (2), it must be that $B'$ ramifies at $\omega$ as well. Because $B'$ is not ramified at all real places of $k$ we may deduce that $\Ram_\infty(B')=\{\omega : \omega \text{ is a real place of $k$ not equal to $\nu$}\}$. Also, because $B'$ satisfies (2) we see that $\Ram_f(B')=\emptyset$. This shows that if $B'$ and $L_i$ satisfy condition (2) of Theorem \ref{thm:CF} then $\Ram(B')=\{\omega : \omega \text{ is a real place of $k$ not equal to $\nu$}\}$. Because a quaternion algebra is completely determined by the primes that ramify in the algebra, we conclude that there is at most one quaternion algebra over $k$ for which the conditions in Theorem \ref{thm:CF} are satisfied for any of the $\Omega_i$ and $L_i$. The theorem now follows from Theorem \ref{thm:CF}.
\end{proof}

\section{A useful lemma}

In this section we prove a lemma which will play an important role in the proofs of our main theorems.

\begin{lemma}\label{lemma:sameinvariants}
Let $\Gamma, \Gamma'$ be arithmetic Fuchsian groups for which the surfaces $\bfH^2/\Gamma, \bfH^2/\Gamma'$ have closed geodesics of length $\ell$. Let $\gamma\in\Gamma$ be the hyperbolic element associated to $\ell$ and $\lambda_\gamma$ the corresponding eigenvalue. Then the invariant trace fields of $\Gamma$ and $\Gamma'$ are equal, and the invariant quaternion algebras of $\Gamma$ and $\Gamma'$ both admit embeddings of the quadratic extension $\Q(\lambda_{\gamma^2})$ of this common invariant trace field.
\end{lemma}
\begin{proof}
Because $\gamma^2$ is contained in the subgroup $\Gamma^2$ of $\Gamma$ generated by squares, which is derived from a quaternion algebra \cite[Chapter 8]{MR}, it follows from Lemma 2.3 of \cite{CHLR} that the invariant trace field of $\Gamma^2$, and hence of $\Gamma$, is $\Q(\lambda_{\gamma^2}+1/\lambda_{\gamma^2})=\Q(\tr(\gamma^2))$. Because $\bfH^2/\Gamma'$ also contains a geodesic of length $\ell$, the geodesic length formula shows that $\Gamma'$ contains an  element $\gamma'$ such that $\tr(\gamma')=\tr(\gamma')$ (up to a sign). In particular this implies that \[\tr(\gamma'^2)=\tr^2(\gamma')-2=\tr^2(\gamma)-2=\tr(\gamma^2),\] from which we conclude that $\Q(\tr(\gamma^2))=\Q(\tr(\gamma'^2))$. Because $\Q(\tr(\gamma^2))$ is the invariant trace field of $\Gamma$ and $\Q(\tr(\gamma'^2))$ is the invariant trace field of $\Gamma'$, this proves the first part of the theorem.

Let $k$ denote the invariant trace field of $\Gamma$ and $\Gamma'$. Let $B\Gamma$ denote the invariant quaternion algebra of $\Gamma$ and $B\Gamma'$ the invariant quaternion algebra of $\Gamma'$. The fields $k(\lambda_{\gamma^2})$ and $k(\lambda_{\gamma'^2})$ embed into $B\Gamma$ and $B\Gamma'$ by \cite[Chapter 8]{MR}, hence the theorem follows from the fact that $k(\lambda_{\gamma^2})\cong \Q(\lambda_{\gamma^2})\cong k(\lambda_{\gamma'^2})$.
\end{proof}

The proof of Lemma \ref{lemma:sameinvariants} also shows the following.

\begin{lemma}\label{lemma:embedlemma}
Let $\Gamma, \Gamma'$ be a arithmetic Fuchsian groups derived from quaternion algebras for which the surfaces $\bfH^2/\Gamma, \bfH^2/\Gamma'$ have closed geodesics of length $\ell$. Let $\gamma\in\Gamma$ be the hyperbolic element associated to $\ell$ and $\lambda_\gamma$ the corresponding eigenvalue. Let $k$ denote the invariant trace fields of $\Gamma$ and $\Gamma'$. Then the invariant quaternion algebras of $\Gamma$ and $\Gamma'$ both admit embeddings of the quadratic extension $k(\lambda_\gamma)$ of $k$.
\end{lemma}


\section{Proof of Theorem \ref{theorem:shortintervals}}

Let $S=\{\ell_1,\dots,\ell_r\}$ be a set of nonnegative real numbers for which $\pi(V,S)\to\infty$ as $V\to\infty$. Let ${\bf H}^2/\Gamma_0$ be an arithmetic hyperbolic surface derived from a quaternion algebra whose length spectrum contains $S$. Let $k$ be the invariant quaternion algebra of $\Gamma_0$ and $B_0$ be the invariant trace field of $\Gamma_0$. For $i=1,\dots,r$ define $L_i=k(\lambda_i)$. Since $\pi(V,S)\to\infty$ as $V\to\infty$ there are infinitely many pairwise non-commensurable arithmetic hyperbolic surfaces derived from quaternion algebras with geodesics of lengths $\{\ell_1,\dots,\ell_r\}$. By Lemma \ref{lemma:embedlemma} the invariant quaternion algebras of these surfaces, which are pairwise non-isomorphic, all admit embeddings of $L_1,\dots,L_r$. This shows, in particular, that there are infinitely many primes of $k$ which are inert in all of the extensions $L_i/k$.

Suppose that $B$ is a quaternion algebra over $k$ which is unramified at a unique real place of $k$, admits embeddings of $L_1,\dots, L_r$ and satisfies $\Ram_f(B)\neq \emptyset$. For each $i=1,\dots, r$, fix a quadratic $\mathcal O_k$-order $\Omega_i\subset L_i$ which contains a preimage in  $L_i$ of $\gamma_i$. It follows from Theorem \ref{thm:selectivity} that every maximal order of $B$ contains conjugates of all of the $\Omega_i$. If $\mathcal O$ is one such maximal order then the arithmetic hyperbolic surface ${\bf H}^2/\Gamma_{\mathcal O}$, which is by definition derived from a quaternion algebra, must have length spectrum containing $S$. Let $V_0$ denote the area of ${\bf H}^2/\Gamma_{\mathcal O}$.

Let $\epsilon>0$ and define $\theta=\frac{8}{3}$ if $r=1$ and $\theta=\frac{1}{2^r}$ if $r>1$. Finally, let $V^{1-\theta+\epsilon} < W < V$. In light of the previous paragraph it suffices to show that for all sufficiently large $V$ one can construct at least $\frac{1}{2^r}\cdot \frac{W}{\log V}$ quaternion algebras $B$ which are ramified at a finite prime of $k$, a unique real place of $k$, admit embeddings of all of the $L_i$ and satisfy $\coarea(\Gamma_\mathcal O)\in (V,V+W)$ where $\mathcal O$ is a maximal order of $B$. Let $\frakp_0$ be a prime of $k$ which is inert in all of the extensions $L_i/k$ (for $i=1,\dots, r$), is unramified in $B_0$ and which satisfies $N(\frakp_0)>13$. Note that such a prime exists because we have already shown that there are infinitely many primes of $k$ which are inert in all of the extensions $L_i/k$. Before continuing we note that because the compact (respectively non-compact) hyperbolic $2$-orbifold of minimal area has area $\pi/42$ (respectively, $\pi/6$), the fact that $N(\frakp_0)>13$ ensures that $V_0\cdot (N(\frakp_0)-1)>1$ (see \cite{K}). We will now construct our quaternion algebras $B$ by choosing primes $\frakp$ of $k$ which are unramified in $B_0$ and inert in all of the extensions $L_i/k$, and then defining $B$ to be the quaternion algebra for which $\Ram(B)=\Ram(B_0)\cup \{\frakp_0,\frakp\}$. As all of the $L_i$ embed into $B_0$ it must be the case that no prime of $\Ram(B_0)$ splits in any of the extensions $L_i/k$. Further, because of the way that we chose $\frakp_0$ and $\frakp$, neither of these primes split in any of the extensions $L_i/k$, hence $B$ admits embeddings of the $L_i$ as desired. If $\mathcal O$ is a maximal order of $B$ then the coarea of $\Gamma_\mathcal O$ is given by \[V_0(N(\frakp_0)-1)(N(\frakp)-1)\] by (\ref{equation:volumeformula}). 

Let $L$ denote the compositum over $k$ of $L_1, L_2,\dots, L_r$. We will show that $[L:k]=2^r$. Suppose to the contrary that $[L:k]=2^s<2^r$. Relabelling the $L_i$ as necessary, we may assume that the compositum over $k$ of $L_1,\dots, L_s$ is $L$. Because $L_r$ is contained in $L$ and $\Gal(L/k)\cong (\mathbb Z/2\mathbb Z)^s$, the Galois correspondence implies that there exist $1\leq i<j\leq s$ such that $L_r$ is contained in the compositum of $L_i$ and $L_j$. Let $\frakq$ be a prime of $k$ which is unramified in $L_i, L_j$ and $L_r$. We claim that $\frakq$ splits in one of these three quadratic extensions of $k$. Indeed, were $\frakq$ inert in all three extensions then the Galois group $\Gal(L_iL_j/k)$ of the compositum of $L_i$ and $L_j$ would have to be cyclic of prime power order \cite[p. 115]{M}, which is not the case since $\Gal(L_iL_j/k)\cong (\mathbb Z/2\mathbb Z)^2$. This shows that there are only finitely many primes of $k$ which do not split in any of $L_i, L_j, L_r$. The proof of Theorem \ref{theorem:csatheorem1} now implies that there are only finitely many quaternion algebras over $k$ which admit embeddings of $L_i, L_j$ and $L_r$, and hence of $L_1,\dots, L_r$. This is a contradiction as we have already seen that there are infinitely many such quaternion algebras. Therefore $[L:k]=2^r$.

We will now employ a version of the Chebotarev density theorem in short intervals due to Balog and Ono \cite{BO}. This theorem shows that the number of primes $\mathfrak P$ of $k$ which are unramified in $L/k$, have $(\mathfrak P,L/k)=(1,\dots,1)\in\Gal(L/k)$ and have $X\leq N(\mathfrak P)\leq X+Y$ is asymptotically \[\frac{1}{2^s}\cdot \frac{Y}{\log X}\] for all sufficiently large $X$ if $\epsilon'>0$ and $X^{1-\theta+\epsilon'}\leq Y\leq X$. Theorem \ref{theorem:shortintervals} now follows from the short intervals version of the Chebotarev density theorem upon setting $c=V_0\cdot (N(\frakp_0)-1)$ and $X=\frac{1}{c}V$.


\section{Proof of Theorem \ref{theorem:finiteness}}

Let $S=\{\ell_1,\dots,\ell_r\}$ be a finite set of nonnegative numbers for which $\pi(V,S)$ is eventually constant and greater than zero. Let $\bfH^2/\Gamma$ be an arithmetic hyperbolic surface derived from a quaternion algebra whose length spectrum contains $S$. Let $k=k\Gamma$ be the invariant trace field of $\Gamma$ and $B=B\Gamma$ be the invariant quaternion algebra of $\Gamma$. For $i=1,\dots, r$, let $\gamma_i$ be the associated hyperbolic element and $\lambda_{\gamma_i}$ be the eigenvalue of the preimage in $\SL_2(\R)$ of $\gamma_i$ for which $|\lambda_{\gamma_i}|>1$.

Suppose that $\bfH^2/\Gamma'$ is an arithmetic hyperbolic surface derived from a quaternion algebra whose length spectrum contains $S$ and which is not commensurable with $\bfH^2/\Gamma$. By Lemma \ref{lemma:sameinvariants}, the invariant trace field of $\bfH^2/\Gamma'$ is also $k$ and the invariant quaternion algebra $B'$ of $\bfH^2/\Gamma'$ admits embeddings of the quadratic extensions $k(\lambda_{\gamma_1}), \dots, k(\lambda_{\gamma_r})$ of $k$. 

Conversely, suppose that $B''$ is a quaternion algebra over $k$ which is unramified at a unique real place of $k$, admits embeddings of $k(\lambda_{\gamma_1}), \dots, k(\lambda_{\gamma_r})$ and is not isomorphic to $B$. For each $i=1,\dots, r$, fix a quadratic $\mathcal O_k$-order $\Omega_i\subset k(\lambda_i)$ which contains a preimage in  $k(\lambda_i)$ of $\gamma_i$. It follows from Theorem \ref{thm:selectivity} that with one possible exception (which can occur only if the narrow class number of $k$ is greater than one), every maximal order of $B''$ contains conjugates of all of the $\Omega_i$ and hence gives rise to an arithmetic hyperbolic surface ${\bf H}^2/\Gamma_{\mathcal O}$ containing closed geodesics of lengths $\ell_1,\dots,\ell_r$. Moreover, such a surface is, by definition, derived from a quaternion algebra.

From the above we deduce that with one possible exception, every isomorphism class of quaternion algebras over $k$ which split at a unique real place of $k$ and admit embeddings of all of the fields $k(\lambda_{\gamma_i})$ will give rise to an arithmetic hyperbolic surface derived from a quaternion algebra with length spectrum containing $S$. Moreover, because these algebras are pairwise non-isomorphic, the associated hyperbolic surfaces are pairwise non-commensurable. Theorem \ref{theorem:finiteness} now follows from Corollary \ref{cor:csacor}.


\section{Proof of Theorem \ref{theorem:existence}}

Fix an integer $n\geq 0$. By Theorem \ref{theorem:csatheorem2} there exist quadratic extensions $L_1,\dots, L_r$ of $\mathbb Q$ such that there are precisely $2^n-1$ quaternion division algebras over $\mathbb Q$ which admit embeddings of all of the $L_i$. Moreover, as was explained in Remark \ref{totallyreal}, we may take these quadratic fields to all be real quadratic fields. The results of \cite[Chapter 12.2]{MR} (see for instance \cite[Theorem 12.2.6]{MR}, which also holds in the context of arithmetic hyperbolic surfaces) show that these real quadratic fields give rise to hyperbolic elements $\gamma_1,\dots,\gamma_r$ of $\PSL_2(\mathbb R)$ and that each of the $2^n-1$ quaternion division algebras gives rise to an arithmetic hyperbolic surface derived from a quaternion algebra containing closed geodesics of lengths $\ell(\gamma_i),\dots,\ell(\gamma_r)$. Here we have used the fact that by Theorem \ref{thm:CF}, every maximal order of these quaternion algebras contains a conjugate of each of the $\gamma_i$. Similarly, the quaternion algebra $\M_2(\mathbb Q)$ admits embeddings of all of these real quadratic fields and gives rise to the hyperbolic surface ${\bf H}^2/\PSL_2(\mathbb Z)$ (whose length spectrum must also contain $\ell(\gamma_i),\dots,\ell(\gamma_r)$). Let $S=\{\ell(\gamma_1),\dots,\ell(\gamma_r)\}$. We have just shown that for sufficiently large $V$ we have that $\pi(V,S)\geq 2^n$. Suppose now that ${\bf H}^2/\Gamma$ is an arithmetic hyperbolic surface derived from a quaternion algebra whose length spectrum contains $S$. Lemma \ref{lemma:sameinvariants} shows that the invariant trace field of this surface is $\mathbb Q$ and that its invariant quaternion algebra admits embeddings of the real quadratic fields $L_1,\dots, L_r$. Recall that two arithmetic hyperbolic surfaces are commensurable if and only if they have isomorphic invariant trace fields and invariant quaternion algebras \cite[Chapter 8.4]{MR}. If ${\bf H}^2/\Gamma$ is not compact then it is commensurable with ${\bf H}^2/\PSL_2(\mathbb Z)$, while if ${\bf H}^2/\Gamma$ is compact its invariant quaternion algebra must be one of our $2^n-1$ quaternion division algebras by Theorem \ref{theorem:csatheorem2}. This shows that ${\bf H}^2/\Gamma$ is commensurable to one of the $2^n$ hyperbolic surfaces constructed above. Theorem \ref{theorem:existence} follows.


\end{document}